\documentclass[]{article}

\usepackage{fancyhdr}
\usepackage{graphicx}
\usepackage{ mathrsfs }
\usepackage{amsmath,amssymb,amsfonts,amsthm}
\usepackage{amsrefs}
\usepackage{mathtools}
\usepackage{tikz}
\usepackage[]{hyperref}
\usepackage{mathtools}
\usepackage{thmtools}

\usepackage{amsmath}
\usepackage{tikz}
\usepackage{pgfplots}
\usepackage[shortlabels]{enumitem}

\usepackage{amsthm}

\theoremstyle{plain} 

\newtheorem{thm}{Theorem}[section]
\newtheorem{corollary}[thm]{Corollary}

\newtheorem{lemma}[thm]{Lemma}

\theoremstyle{remark}     
\newtheorem{remark}[thm]{Remark}

\renewcommand{\Re}{\operatorname{Re}}

\usepackage{color}
\newcommand{\kf}[1]{{\color{magenta}{#1}}}

\begin{document}

\title{Spectral and dynamical invariants of Hecke triangle groups via transfer operators}
\author{Ksenia Fedosova}
\date{}
\maketitle
\begin{abstract}
In this paper, we consider Hecke triangle groups $\Gamma_w$ for $w>2$ and associated infinite-volume orbifolds $\Gamma_w \backslash \mathbb{H}$. We show that the 
Selberg zeta function $Z_{\Gamma_w}(s)$ can be approximated for 
$s \in \mathbb{C} \setminus \frac{1}{2}(1-2 \mathbb{N}_0)$  
by determinants of finite-dimensional matrices with an explicitly computed error term 
that decays exponentially as the matrix size increases. 
As an application, we evaluate the Hausdorff dimensions of Hecke triangle groups with high precision, 
explicitly compute the values of the corresponding Ruelle zeta functions at zero, and obtain estimates on orders of trivial zeroes of the Selberg zeta function. 
\end{abstract}
\section{Introduction}

The Hecke triangle group $\Gamma_w$ for $w>0$ is the subgroup of $\mathrm{PSL}_2(\mathbb{R})$ generated by the elements
\begin{equation}
T_w:=\left[\begin{array}{cc}
1 & w \\
0 & 1
\end{array}\right] \quad \text { and } \quad S:=\left[\begin{array}{cc}
0 & 1 \\
-1 & 0
\end{array}\right].    
\end{equation}
For admissible values of $w$, the group $\Gamma_w$ acts on the hyperbolic plane $   \mathbb{H} = \{ z \in \mathbb{C} \, : \, \operatorname{Im}(z)>0  \}$  with hyperbolic orbifold quotients. In particular, $\Gamma_1$ coincides with the modular group $\operatorname{PSL}_2(\mathbb{Z})$. However, for $w>2$, the case of interest of this article, the quotient $\Gamma_w \backslash \mathbb{H}$ is an infinite-volume orbifold with one cusp and one conical singularity.  We define the \textit{limit set} $\Lambda_w$ of $\Gamma_w$ to be the set of limit points of any fixed orbit of $\Gamma_w$. For $w>2$, $\Lambda_w$ is a fractal subset of $\mathbb{R} \cup \{ \infty \}$, and we define $\delta_w$ as its \textit{Hausdorff dimension}.

One of the goals of this article is to obtain $\delta_w$ with rigorously controlled high precision, achieved through approximation of the Selberg zeta function with an explicit exponentially decaying error bound. 
Such an approximation of the Selberg zeta function (extended to its meromorphic continuation domain) enables the \emph{explicit evaluation} of the Ruelle zeta function at $s=0$ using the transfer operator approach. 
To the best of the author’s knowledge, these computations represent the first instance where transfer operators are used to study Ruelle zeta values in a geometric setting. In addition, the approach allows us to derive  estimates on the orders of the trivial zeros of the Selberg zeta function.

For a hyperbolic orbifold $\Gamma \backslash \mathbb{H}$, the \textit{Selberg zeta function} is originally defined for $\text{Re}(s)$ sufficiently large by 
\begin{equation}\label{eq:sssppppp2p2p2pp2p2}
Z_{\Gamma}(s)=\prod_{[\gamma]} \prod_{k=0}^{\infty} (1- e^{-(s+k) \ell(\gamma)}),
\end{equation}
where $[\gamma]$ runs over conjugacy classes of primitive hyperbolic elements of $\Gamma$ and $\ell(\gamma)$ denotes the length of a closed geodesic in $\Gamma \backslash \mathbb{H}$ associated to $\gamma \in \Gamma$. In many settings, the Selberg zeta function admits a meromorphic continuation to the complex plane, and its zeros and poles encode spectral information of the orbifold -- specifically, the spectrum of the Laplace operator on $\Gamma \backslash \mathbb{H}$. For the group $\Gamma_w$ with $w>2$, the Selberg zeta function $Z_{\Gamma_w}(s)$ has exactly one zero in the half-plane $\Re(s) > \tfrac{1}{2}$, located at $s = \delta_w$. This zero corresponds to the \textit{base eigenvalue} $\delta_w(1 - \delta_w)$ of the Laplacian on $\Gamma_w \backslash \mathbb{H}$.

In his seminal paper \cite{mayer1991thermodynamic}, Mayer introduced a thermodynamical approach to the Selberg zeta function for the modular surface $\operatorname{PSL}_2(\mathbb{Z}) \backslash \mathbb{H}$. He defined a so-called \textit{transfer operator} $\mathcal{L}_{s} $ such that, for $\text{Re}(s)$ sufficiently large, the Selberg zeta function associated to the modular surface was realized as the Fredholm determinant of the operator $1- \mathcal{L}_s$,
\begin{equation}
Z_{\operatorname{PSL}_2(\mathbb{Z}) \backslash \mathbb{H}}(s)=\operatorname{det}(1-\mathcal{L}_{s}).
\end{equation}
Mayer then showed that $s \mapsto \mathcal{L}_{s}$ admits a meromorphic continuation to $\mathbb{C}$, that allowed to obtain the meromorphic continuation of the Selberg zeta function. This result was extended to other groups and different settings, see e.g. \cite{fedosova2020meromorphic, moller2013period, pohl2015thermodynamic, pohl2016symbolic, guillope2004selberg, pollicott1991some, pollicott1997remarkable } and references therein. In this paper, we analyze the extension of these results to Hecke triangle groups with  $w > 2$ and the  transfer operator $\mathcal{L}_s$ introduced in \cite{naud2018fractal}.

However, we introduce a slightly different approach from Mayer. Instead of studying the meromorphic continuation of the operator $\mathcal{L}_s$, we directly prove the meromorphic continuation of the determinant $\det(1-\mathcal{L}_s)$. To do this, we use a sequence of meromorphic functions $F_N(s)$, first introduced in \cite{soares2022hecke}, which (as we will show in Theorem~\ref{thm:main_theorem_1}) approximate $\det(1-\mathcal{L}_s)$ uniformly on compact subsets of $\mathbb{C} \setminus \tfrac{1}{2}(1-2\mathbb{N}_0)$ as $N \to \infty$. In \cite{soares2022hecke}, only the asymptotic convergence of $F_N(s)$ to the Selberg zeta function was established, and only in the restricted half-plane $\Re(s) > \tfrac{1}{2}$. In this paper, we extend the method to the full holomorphic domain $\mathbb{C} \setminus \tfrac{1}{2}(1-2\mathbb{N}_0)$, where the Selberg zeta function exhibits its most interesting behavior. Second, instead of merely asymptotic convergence, we provide an explicit and exponentially decaying error bound (Theorem~\ref{thm:main_theorem_1}).  This allows for high-precision computation of Hausdorff dimensions of the boundaries of Hecke triangle groups with the help of the bisection method, improving the three digits previously reported in \cite{sarnak2008spectrum} and \cite{soares2022hecke}  to at least fifty.

In order to formulate the main theorem, define
\begin{equation}\label{def:Cs}
\begin{split}
    C(s) = \sup_{n \in 2 \mathbb{N}_0}  |\zeta(2s+n)|, \quad s \in \mathbb{C} \setminus \frac{1}{2}(1-2\mathbb{N}_0),
\end{split}
\end{equation}
where $\zeta$ denotes the Riemann zeta function. 
For $s \in \mathbb{C} \setminus \frac{1}{2}(1-2 \mathbb{N}_0)$ and $N \in \mathbb{N}$, we define $F_N(s)$ to be a determinant of the $N \times N$ matrix $(1 - L(s))$,
\begin{equation}\label{eq:definition_FN}
F_N(s) = \det (1-L(s)),    
\end{equation}
where the $(i,j)$-th entry $\ell_{ij}(s)$ of $L(s)$ is equal to 
\begin{equation}\label{eq:elsijeven}
\ell_{ij}(s) = 
2     \sqrt{\frac{j+1}{i+1}} \cdot   \frac{\zeta(2 s+i+j)}{w^{2 s+i+j}}\binom{2 s+i+j-1}{i}  \quad \text{for } i + j \text{ even}, 
\end{equation}
and 
\begin{equation}\label{eq:elsijodd}
    \ell_{ij}(s) = 0 \quad \text{for } i + j \text{ odd}.
\end{equation}
We additionally introduce the functions $P_N(s)$ and $Q(s)$ defined for $s \in \mathbb{C} \setminus \frac{1}{2} ( 1 - 2\mathbb{N}_0) $, $w>2$ and $N \in \mathbb{N}$ as 
\begin{equation}\label{eq:PNs}
P_N(s) = \left(\frac{2}{w}\right)^N (N+1)^{1/2} \cdot  w C(s) \frac{  2^{  \lceil 2|s| \rceil    } }{w^{2 \Re(s)}}     \operatorname{Li}_{-\frac{1}{2}} \left(\frac{2}{w}\right)   \left[  1  +    \sqrt{2} \operatorname{Li}_{-\frac{1}{2}}\left( \frac{2}{w}\right) \right],
\end{equation}
and 
\begin{equation}\label{eq:QNs}
    Q(s) =  2 C(s)    \frac{w^{2 | s|-2 \Re(s) }}{(w-1)^{2 | s| -1}}  \cdot \operatorname{Li}_{-\frac{3}{2}}\left(\frac{1}{w-1}\right) ,
\end{equation}
where $\operatorname{Li}$ denotes the polylogarithm \cite[25.12.10]{NIST}. We note that for a fixed value of $s \in \mathbb{C} \setminus  \frac{1}{2} ( 1 - 2 \mathbb{N}_0) $ and $w>2$, we have 
\begin{equation}
    P_N(s) = O(N^{1/2} (w/2)^{-N}),  \quad N \to \infty. 
\end{equation}

One of the main results of this paper is the following:
\begin{thm}\label{thm:main_theorem_1}
For $s \in \mathbb{C} \setminus \frac{1}{2}(1-2\mathbb{N}_0)$ and $N \in \mathbb{N}$, let $F_N(s)$ be as in \eqref{eq:definition_FN}. The Selberg zeta function $Z_{\Gamma_w}(s)$ admits a meromorphic continuation to $\mathbb{C} \setminus \frac{1}{2}(1-2 \mathbb{N}_0)$, and in that domain, it can be  approximated by $F_N(s)$ with the error term that decays exponentially fast as $N \to \infty$:
    \begin{equation}\label{eq:approximation_of_selberg_zeta}
    \begin{split}
        |  Z_{ \Gamma_w}(s)  & - F_N(s) |  \le  P_N(s)   e^{P_N(s)+Q(s)+1 }
    \end{split}
    \end{equation}
    for $P_N(a)$ and $Q(s)$ as in \eqref{eq:PNs} and \eqref{eq:QNs}, respectively.
 \end{thm}

As a corollary of this theorem, we can obtain approximate values of the first eigenvalue of the Laplace operator of $\Gamma_w \setminus \mathbb{H}$;  see Section \ref{sec:haus_calc} and Table \ref{tab:hausdorff}. 
\begin{table}[h!]
\centering
\begin{tabular}{c|c}
$w$ & Hausdorff dimension of $\delta(\Gamma_w)$ \\ \hline 
$3$ & $0.75194008038202898753355087134612238565071248482239$ \\
$4$ & $0.68367105376320840963103084607448961221631125476496$ \\
$5$ & $0.64665638884984061955006624797665443932208623918330$ \\
$6$ & $0.62296896860108742758578970214133058127260612238989$ \\
$8$ &  $0.59395687467303202626541162773916197787885310359836$ \\
$10$ & $0.57660658272884532239298217889172324836908688431275$ \\
$16$ & $0.55011004182730371669178285114466116320309677135534$ \\
$40$ & $0.52182151093148260901879103287698690165007405447213$ \\
$100$ & $0.50927941737580653723736709527094585385489171074337$
\end{tabular}
\caption{Hausdorff dimensions corresponding to different $w$.}
\label{tab:hausdorff}
\end{table}

Additionally, we apply Theorem \ref{thm:main_theorem_1} to study the Ruelle zeta function defined for $\Re(s) > 1$ and $w>2$ as follows:
\begin{equation}\label{eq:def_Ruelle_zeta}
    R_{\Gamma_w}(s) = Z_{\Gamma_w}(s) / Z_{\Gamma_w}(s+1).
\end{equation}
From the meromorphic continuation of $Z_{\Gamma_w}(s)$ as a function  of $s \in \mathbb{C}$, we obtain a meromorphic continuation of $R_{\Gamma_w}(s)$ and can evaluate the latter at $s=0$. 
\begin{corollary}\label{sec:ruelle_hecke_triangle}
    For $w>2$, the Ruelle zeta $R_{\Gamma_w}(s)$ function, defined by \eqref{eq:def_Ruelle_zeta}, admits a meromorphic continuation to $s \in \mathbb{C}$ that is regular at $s=0$, and
    \begin{equation}
        R_{\Gamma_w}(0) = 2.
    \end{equation}
\end{corollary}
The values of Ruelle zeta functions $R(s)$ at $s=0$ play a prominent role in the context of Fried’s conjecture (see \cite{fried1986analytic,  dyatlov2017ruelle, moscovici1991r, mueller2020fried, shen2017analytic, jorgenson2024functional, dang2020fried} and references therein). Previous works that evaluate $R(s)$ or its order of vanishing at $s=0$ in geometric contexts have typically relied on the Selberg trace formula or microlocal analysis and, as far as the author is aware, the transfer operator approach for the explicit evaluations of $R(0)$ was not previously used. 

As another corollary of Theorem \ref{thm:main_theorem_1}, we can study the trivial zeros of the Selberg zeta function and evaluate the order of their degrees:
 
\begin{corollary}
    \label{lem:skdjfjfjfjsllfksdjlkflskdflksdflksjdfjsdjfjfjjj}
 Let $w > 2$ and let $m \in \mathbb{N}$. The function $Z_{\Gamma_w}(s)$ has a zero at the point $s = -m$ with the degree at least $m$, but no bigger than $2m+1$.
\end{corollary}

\subsubsection*{Acknowledgments} 
The author is grateful to Kim Klinger-Logan and Jørgen Lye for helpful discussions and to Axel Sorge for assistance with calculations. KF is partially funded by the Deutsche Forschungsgemeinschaft (DFG, German Research Foundation) under Germany's Excellence Strategy EXC 2044-390685587, Mathematics M\"unster: Dynamics-Geometry-Structure.
\section{Preliminaries}

\subsection{Functional analysis}
In this section, we collect the necessary information about trace class operators and provide a crude estimate of their nuclear norms. Let $\mathscr{H}$ be a Hilbert space equipped with the inner product $\langle \cdot, \cdot \rangle$. Suppose that $A: \mathscr{H} \to \mathscr{H}$ is a compact operator defined via its action on an orthonormal basis $\{ \psi_j \}_{j \in \mathbb{N}_0}$ of $\mathscr{H}$:
\begin{equation}\label{def:action}
    \langle \psi_i, A \psi_j \rangle = \ell_{ij}.
    \end{equation}
Denote by $\mu_m(A)$, $m \in \mathbb{N}$, the singular values of $A$, arranged in decreasing order \cite[ p. 49, (2.1)]{gohberg2012traces}.  
They satisfy the min-max principle \cite[ p. 49,  (2.2)]{gohberg2012traces}: for $m \in \mathbb{N}$, we have 
\begin{equation}
   \mu_m(A)=\min _{\substack{V \subset \mathscr{H} \\ \operatorname{dim}(V)=m-1}} \max _{\psi \in V^{\perp}} \frac{\|A \psi\|}{\|\psi\|},
\end{equation}
where $\| \cdot \| $ denotes the operator norm.
To estimate $\mu_m(A)$ from above, consider an $(m-1)$-dimensional  subspace   $V$ of $\mathscr{H}$ spanned by $ \{ \psi_j \}_{j=0}^{m-2} $, then (fixing a minor misprint in \cite[(16)]{soares2022hecke})
\begin{equation}\label{eq:sing_value_est}
    \mu_m(A) \le \max_{ \psi \in  V^\perp }  \frac{\| A \psi \| }{ \| \psi \|} \le \sum_{j = m-1}^\infty  \|A \psi_j  \|.
\end{equation}
We recall that the trace norm of the operator $A$ is by definition equal to the sum of singular values of~$A$:
\begin{equation}\label{eq:trace_norm}
\| A \|_1 = \sum_{m=1}^\infty \mu_m(A).
\end{equation}
If $\| A \|_1 < \infty$, then $A$ is called a nuclear operator. 
We can evaluate the right hand side of \eqref{eq:trace_norm} as follows:
\begin{equation}\label{eq:crude_trace_estimate2}
\| A \|_1    \stackrel{\eqref{eq:sing_value_est}}{\le} \sum_{m=1}^\infty \sum_{j =m-1}^\infty \|A \psi_j \| = \sum_{j=0}^\infty (j+1) \| A \psi_j \| \le \sum_{i=0}^\infty \sum_{j=0}^\infty (j+1)  | \ell_{ij} |.
\end{equation}
If $A$ is of trace class, then so is $A^*$ \cite[p. 51, Corollary 2.2]{gohberg2012traces}, and given that $\| A^* \|_1 = \| A \|_1$, we obtain 
\begin{equation}\label{eq:crude_trace_estimate3}
\| A \|_1 \le     \sum_{i=0}^\infty \sum_{j=0}^\infty (j+1)  | \ell_{ji} | = \sum_{i=0}^\infty \sum_{j=0}^\infty (i+1)  | \ell_{i j} |.
\end{equation}
The estimates are in no way sharp; e.g., for an operator $A$ defined by \eqref{def:action} with  
\begin{equation}
    \ell_{ij} = \begin{cases}
        (j+1)^{-2}, \quad & i = j, \\
        0, \quad & \text{else},
    \end{cases}
\end{equation}
we have $A=|A|$ and hence $\| A\|_1 = \zeta(2) $. On the other hand,  we cannot even deduce from either \eqref{eq:crude_trace_estimate2} or~\eqref{eq:crude_trace_estimate3} that the operator $A$ is of trace class, because the bounds read $\| A\|_1 \le \infty$.

\subsection{Transfer operator}

In this section, we recall the construction of the transfer operator $\mathcal{L}_s$ from \cite{soares2022hecke}. 
Let $\mathbb{D}=\{|z|<1\}$, and define the Bergman space as 
\begin{equation}
H^2(\mathbb{D} ):=\{f: \mathbb{D} \rightarrow \mathbb{C} \text { holomorphic } \mid\|f\|<\infty\} 
\end{equation}
with $L^2$-norm given by $\|f\|^2:=\int_{\mathbb{D}} |f(z)|^2 \operatorname{dvol}(z)$. For $\text{Re}(s) > 1/2$, following \cite[(20)-(22)]{soares2022hecke}, we introduce  the \textit{transfer operator} $\mathcal{L}_{s}$ defined by its action on functions $f \in H^2(\mathbb{D} )$ as 
\begin{equation}\label{eq:definition_transfer_operator_as44f}
\mathcal{L}_{s} f(z)=\sum_{n \in \mathbb{Z} \backslash\{0\}} (|n|w)^{-2 s} e^{-2s \log(1+\frac{z}{n w})} f\left(- \frac{1}{z+nw} \right)  , \quad z \in \mathbb{D}.
\end{equation}
Then, \cite[Proposition 2.1]{soares2022hecke} implies that $\mathcal{L}_{s} $, defined in this way, is a self-map and is a nuclear operator. Moreover, for $\text{Re}(s) > 1/2$, the Selberg zeta function can be realized as the Fredholm determinant of the operator $1- \mathcal{L}_s$ \cite[(36)]{soares2022hecke}
\begin{equation}
Z_{\Gamma_w}(s)=\operatorname{det}(1-\mathcal{L}_{s}).
\end{equation}

\begin{remark}\label{rem:meromorphic_continuation}
The next step in \cite{naud2018fractal} as well as in  \cite{soares2022hecke}, which is standard in the literature, is to introduce a certain finite rank operator $\mathcal{F}_s$ that is holomorphic as a function of $s$, and a bounded operator $\Psi$ such that (see \cite[Section 2.3]{naud2018fractal} or \cite[(43)]{soares2022hecke})
\begin{equation}
\mathcal{L}_s = \mathcal{F}_s + \mathcal{L}_{s+1/2} \circ \Psi .
\end{equation}
From the expression above, if $\mathcal{L}_s$ defines a meromorphic function in a half-plane $\operatorname{Re}(s) > s_0$ for $s_0 \in \mathbb{R}$, then it also defines a meromorphic function in the larger domain $\operatorname{Re}(s) > s_0 - 1/2$. Iterating this argument  yields a meromorphic continuation of $\mathcal{L}_s$ to the entire complex plane (see, e.g., \cite[proof of Corollary 1.3]{soares2022hecke}). However, we will not rely on this result here. 
\end{remark}

In \cite[(26)]{soares2022hecke}, the author wrote an explicit representation of the operator $\mathcal{L}_s$ in an orthonormal basis $\{ \psi_j \}_{j \in \mathbb{N}_0 }$ for the  Bergman space $H^2(\mathbb{D})$,
\begin{equation}\label{eq:psidefddasds}
\psi_j(z)=\sqrt{\frac{j+1}{\pi}} z^j, \quad j \in \mathbb{N}_0,
\end{equation}
and wrote an explicit action of $\mathcal{L}_s$ in this basis \cite[Proposition 4.1]{soares2022hecke}:
\begin{equation}\label{def:mathcalLs}
\left\langle \psi_i, \mathcal{L}_s \psi_j \right\rangle=   \ell_{ij}(s), \quad \operatorname{Re}(s)>1 / 2,
\end{equation}
where $\ell_{ij}(s)$ is as in \eqref{eq:elsijeven} and \eqref{eq:elsijodd}.

Further, \cite{soares2022hecke} constructed a family of finite rank operators 
\begin{equation}
\operatorname{Pr}_N \circ \mathcal{L}_s \circ \operatorname{Pr}_N,
\end{equation}
where $\operatorname{Pr}_N$ denotes the projection onto $ \{ \psi_j \}_{j=0}^{N-1}$. As \cite{soares2022hecke} showed, such operators approximate~$\mathcal{L}_s$ exponentially fast   in the sense of trace norm for  $\operatorname{Re}(s)>1 / 2$.

\section{Proof of Theorem \ref{thm:main_theorem_1}}

For $i,j \in N_0$, each of $\ell_{i j}(s)$, defined by \eqref{eq:elsijeven} and \eqref{eq:elsijodd}, is either a meromorphic function in $s \in \mathbb{C}  $ with the only pole at $\frac{1-i-j}{2}$ for $i+j$ even, or a constant function for $i+j$ odd. Let $\psi_j$, $j \in \mathbb{N}_0$ be as in \eqref{eq:psidefddasds}. Let $\mathcal{W}(s)$ and $\mathcal{W}_N(s)$ the operators defined by their actions on $\{ \psi_j \}_{j \in \mathbb{N}_0 }$  as 
\begin{equation}\label{def:mathcalWs}
\left\langle \psi_i, \mathcal{W} (s) \psi_j \right\rangle=   \ell_{ij}(s), \quad s \in \mathbb{C} \setminus \frac{1}{2} (1-2 \mathbb{N}_0)
\end{equation}
and 
\begin{equation}\label{def:mathcalWs}
\left\langle \psi_i, \mathcal{W}_N (s) \psi_j \right\rangle= \delta_{0 \le i, j \le N-1} \cdot  \ell_{ij}(s), \quad s \in \mathbb{C} \setminus \frac{1}{2} (1-2 \mathbb{N}_0),
\end{equation}
where $\delta_{0 \le i, j \le N-1}$ is equal to one for $  0 \le i, j \le N-1 $ and zero otherwise. 

Note that meromorphy of $s \mapsto \ell_{ij}(s)$ is not enough to conclude that $\mathcal{W}(s)$ is a meromorphic family of operators and to imply that $\mathcal{W} (s)$ coincides with meromorphic continuation of $\mathcal{L}_s$ from Remark~\ref{rem:meromorphic_continuation}.

Before we proceed to the proof of Theorem \ref{thm:main_theorem_1}, we will need an additional lemma. We aim to prove that for $s \in \mathbb{C} \setminus \frac{1}{2}(1- 2 \mathbb{N}_0)$, the operator $\mathcal{W}(s)$ is of trace class, and that operators $\mathcal{W}_N(s)$ approximate $\mathcal{W}(s)$ exponentially fast in $N$ in the sense of trace norm. We follow the general route of  \cite[Lemma 4.2]{soares2022hecke}, but provide explicit estimates instead of asymptotic ones, perform evaluations for all $s \in \mathbb{C} \setminus \frac{1}{2}(1- 2 \mathbb{N}_0)$ instead of the region $\Re(s) > 1/2$ and show that the convergence is, in fact, uniform in $s$. 

\begin{lemma}\label{lemma:approxL} Let $s \in \mathbb{C} \setminus \frac{1}{2}(1- 2\mathbb{N}_0)$ and let $N \in \mathbb{N}_0$. Let $P_N(s)$ and $Q(s)$ be as in \eqref{eq:PNs} and \eqref{eq:QNs}, respectively. The operator $\mathcal{W}(s)$ is a trace class operator with 
\begin{equation}\label{eq:trace_norm_estimates}
    \| \mathcal{W}(s) \|_1  \le Q(s).
\end{equation}
Moreover, the operators $\mathcal{W}_N(s)$ approximate $\mathcal{W}(s)$ exponentially fast in the trace norm as a function of~$N$:
\begin{equation}\label{eq:trace_norm_of_difference}
\begin{split}
\left\| \mathcal{W}_N(s) - \mathcal{W}(s) \right\|_1 &\le P_N(s).
\end{split}
\end{equation}
\end{lemma}

\begin{proof}[Proof of Lemma \ref{lemma:approxL}]
We begin the proof by analyzing the definition \eqref{eq:elsijeven} and \eqref{eq:elsijodd} of $\ell_{ij}(s)$. Let $s \in K$ for a compact subset $K$ of $\mathbb{C} \setminus \frac{1}{2}(1-2\mathbb{N}_0)$. Note that for $s \in \mathbb{C}$ and $i,j \in \mathbb{N}_0$, we have 
\begin{equation}
\left|     \binom{2s+i+j-1}{i} \right| \le      \binom{| 2s|+i+j-1}{i}.
\end{equation}
The inequality can be shown by comparing individual factors. 
To examine the contribution of the Riemann zeta function to $\ell_{ij}(s)$, let $n_0 = n_0(K)$ be the smallest natural number such that $2 \Re(s) + n_0>1$ for $s \in K$. Then for $n \in \mathbb{N}_0$, we have 
\begin{equation}
    |\zeta(2 s + n)| \le \max \{  \max_{0 \le n \le n_0} |\zeta(2 s + n)|,  \sup_{n > n_0} |\zeta(2 s + n)| \}.
\end{equation}
Moreover, for $n > n_0$, the argument $2 s + n$ is in domain of the convergence of the series defining the Riemann zeta function, and hence 
\begin{equation}
    \sup_{n > n_0} |\zeta(2 s + n)| \le \zeta( 2 \Re(s) + n_0 ),
\end{equation}
that implies that $C(s)$ is a function that is bounded uniformly on $s \in K$. 
Thus, we have 
\begin{equation}\label{eq:estlijsss}
|    \ell_{ij}(s) | \le 2 C(s) \sqrt{\frac{j+1}{i+1}} \binom{2 |s|+i+j-1}{i} \frac{1}{w^{2 \Re(s) + i + j}}.
\end{equation}
To show \eqref{eq:trace_norm_estimates}, we recall that by the binomial series expansion,
\begin{equation}\label{eq:oopjjjjjdf2}
\begin{split}
    \sum_{i=0}^\infty w^{-i} \binom{2 | s| + i+j  -1}{i}  & =  \left(\frac{w}{w-1}\right)^{2 | s| +j},
\end{split}
\end{equation}
and by the definition of the polylogarithm,
\begin{equation}\label{eq:oopjjjjjdf3}
        \sum _{j=0}^{\infty }   \frac{    (j+1)^{3/2}    }{(w-1)^j}      = (w-1)   \operatorname{Li}_{-\frac{3}{2}} \left(\frac{1}{w-1} \right).
\end{equation}
Thus, 
\begin{equation}\label{eq:oopjjjjjdf}
\begin{split}
 \left\| \mathcal{W} (s) \right\|_1  & \stackrel{\eqref{eq:crude_trace_estimate2}}{\le}  \sum_{i=0}^{\infty} \sum_{j=0}^
{\infty} (j+1) |\ell_{ij} (s) | \\
&\stackrel{\eqref{eq:estlijsss}}{\le} 2 C(s)  \sum_{i=0}^{\infty} \sum_{j=0}^
{\infty}  (j+1)^{3/2} \binom{2 |s|+i+j-1}{i} \frac{1}{w^{2 \Re(s) + i + j}} \\
&\stackrel{ \eqref{eq:oopjjjjjdf2}  }{\le} 2 C(s)  \sum_{j=0}^
{\infty}  (j+1)^{3/2} \frac{1}{w^{2 \Re(s) + j}} \left(\frac{w}{w-1}\right)^{2 | s| +j} \\
& \stackrel{\eqref{eq:oopjjjjjdf3}}{=} 2 C(s)  \frac{ w^{2 |s| - 2 \Re(s)} }{(w-1)^{2 |s|-1}}    \operatorname{Li}_{-\frac{3}{2}} \left(\frac{1}{w-1} \right), \\
\end{split}
\end{equation}
that finishes the proof of \eqref{eq:trace_norm_estimates}. 

Now, to bound the difference between $\mathcal{W}(s)$ and $\mathcal{W}_N(s)$, we want to use \eqref{eq:crude_trace_estimate3} and the definitions of the former to write 
\begin{equation}\label{eq:difference_between_two_oper_adddser}
\begin{split}
 \left\| \mathcal{W} (s)-\mathcal{W}_N (s)\right\|_1  & \le \left(  \sum_{j=N}^\infty \sum_{i=0}^\infty + \sum_{i=N}^{\infty} \sum_{j=0}^
{N-1}\right) (i+1) |\ell_{ij} (s)|.  \\
\end{split}
\end{equation}
\begin{remark}
The inequality \eqref{eq:crude_trace_estimate2} allows for an easier evaluation of $\left\| \mathcal{W} (s) \right\|_1$ than \eqref{eq:crude_trace_estimate3} because of the closed formula \eqref{eq:oopjjjjjdf2}. However, for the estimate of $\left\| \mathcal{W}_N (s)-\mathcal{W} (s) \right\|_1$, it appears that \eqref{eq:crude_trace_estimate3} is more appropriate; if we were to repeat the proof with \eqref{eq:crude_trace_estimate2}, the final estimate would have the order of decay of $N^{3/2} (2/w)^N$ instead of asymptotically sharper $N^{1/2} (2/w)^N$.    
\end{remark}
We use a cruder estimate for the binomial coefficients \cite[p. 1264]{soares2022hecke}
\begin{equation}
    \binom{2 | s|+i+j-1}{i} \le 2^{ \lceil 2 |s| \rceil +i+j-1 },
\end{equation}
from which we obtain 
\begin{equation}
        |\ell_{ij} (s)| \le  C(s) \frac{  2^{  \lceil 2|s| \rceil  + i + j  } }{w^{2 \Re(s)+ i+j}} \sqrt{\frac{j+1}{i+1}},
    \end{equation}
and thus 
\begin{equation}\label{eq:skjkdjdjdkkj1}
    \begin{split}
 \left\| \mathcal{W}_N(s) - \mathcal{W}(s) \right\|_1 & \le  C(s) \frac{  2^{  \lceil 2|s| \rceil    } }{w^{2 \Re(s)}}   \left(  \sum_{j=N}^\infty \sum_{i=0}^\infty + \sum_{i=N}^{\infty} \sum_{j=0}^
{N-1}\right)  (\frac{  2 }{w})^{   i + j  } \sqrt{  (i+1) (j+1)}  \\
 & \le  C(s) \frac{  2^{  \lceil 2|s| \rceil    } }{w^{2 \Re(s)}}   \left(  \sum_{j=N}^\infty \sum_{i=0}^\infty + \sum_{i=N}^{\infty} \sum_{j=0}^
{\infty}\right)  (\frac{  2 }{w})^{   i + j  } \sqrt{     (i+1)    (j+1) }  \\
& = 2  C(s) \frac{  2^{  \lceil 2|s| \rceil    } }{w^{2 \Re(s)}}    \sum_{j=N}^\infty \sum_{i=0}^\infty   (\frac{  2 }{w})^{   i + j  } \sqrt{           (i+1)  (j+1)    } . \\
    \end{split}
\end{equation}
By the definition of the polylogarithm, we have 
\begin{equation}\label{eq:skjkdjdjdkkj2}
\sum_{j=N}^\infty  \sum_{i=0}^\infty  \sqrt{      (i+1)   (j+1)    } ( \frac{2}{w} )^{i+j} = \frac{w}{2}  \operatorname{Li}_{-\frac{1}{2}} (\frac{2}{w})  \sum_{j=N}^\infty  (\frac{2}{w})^j \sqrt{j+1}.
\end{equation}
We evaluate the sum in the right hand side of the previous expression as follows:
\begin{equation}\label{eq:skjkdjdjdkkj3}
\begin{split}
\sum_{j=N}^\infty (\frac{2}{w})^j (j+1)^{1/2} & = \sum_{j'=0}^\infty (\frac{2}{w})^{j'+N} (j'+N+1)^{1/2} \\
&= (N+1)^{1/2} (\frac{2}{w})^{N} \sum_{j'=0}^\infty (\frac{2}{w})^{j'} (\frac{j'}{N+1}+1)^{1/2} \\
& \le (N+1)^{1/2} (\frac{2}{w})^{N}  \left( 1 +  \sum_{j'=1}^\infty (\frac{2}{w})^{j'} (2 j')^{1/2}  \right)  \\
& =(\frac{2}{w})^N (N+1)^{1/2} \left[  1  +    \sqrt{2} \text{Li}_{-\frac{1}{2}}( \frac{2}{w}) \right] . \\
\end{split}
\end{equation}
Combining \eqref{eq:skjkdjdjdkkj1}, \eqref{eq:skjkdjdjdkkj2} and \eqref{eq:skjkdjdjdkkj3}, we obtain 
\begin{equation}
\begin{split}
 \left\| \mathcal{W} (s)-\mathcal{W}_N (s)\right\|_1 
 \le (\frac{2}{w})^N (N+1)^{1/2} \cdot  w C(s) \frac{  2^{  \lceil 2|s| \rceil    } }{w^{2 \Re(s)}}     \operatorname{Li}_{-\frac{1}{2}} (\frac{2}{w})   \left[  1  +    \sqrt{2} \text{Li}_{-\frac{1}{2}}( \frac{2}{w}) \right]
\end{split}
\end{equation}
that finishes the proof.
\end{proof}

Now, we are ready to prove Theorem  \ref{thm:main_theorem_1}:
\begin{proof}[Proof of Theorem \ref{thm:main_theorem_1}]
By the lemma above, $\mathcal{W}(s)$ is a trace class operator, and we can define, for every $s \in \mathbb{C} \setminus \frac{1}{2} (1-2\mathbb{N}_0)$, the Fredholm determinant of $1-\mathcal{W}(s)$, 
\begin{equation}\label{eq:definition_F}
 F(s) = \operatorname{det} (1- \mathcal{W}(s) ).   
\end{equation}
 Note that 
\begin{equation}\label{eq:FN_definition}
    F_N(s) = \det(1 - \mathcal{W}_N(s)).
\end{equation}
Note that for each fixed $N \in \mathbb{N}_0$, the function $F_N(s)$ 
is holomorphic in \( s \in \mathbb{C} \setminus \frac{1}{2}(1 - 2\mathbb{N}_0) \) as a finite sums of products of finitely many holomorphic functions.
We have 
    \begin{equation}
    \begin{split}
        |   F_N(s) - F(s) | &=     | \det(1- \mathcal{W}(s) ) - \det(1- \mathcal{W}_N(s) ) |  \\
        &\le \| \mathcal{W}(s)  - \mathcal{W}_N(s) \|_1 \cdot \exp ( \max \{ \| \mathcal{W}(s) \|_1, \| \mathcal{W}_N(s) \| _1 \} + 1  ).
    \end{split}
    \end{equation}
    The first factor is evaluated in \eqref{eq:trace_norm_of_difference}. To evaluate the second factor, note 
    \begin{equation}
        \| \mathcal{W}_N(s) \| _1 \le        \| \mathcal{W}_N(s) \| _1 +  \| \mathcal{W}_N(s) - \mathcal{W}(s) \| _1
    \end{equation}     and hence 
\begin{equation}\label{estimate_on_F}
    |   F_N(s) - F(s) | \le \| \mathcal{W}_N(s) - \mathcal{W}(s) \|_1 \cdot \exp ( \| \mathcal{W}(s) \|_1  +  \| \mathcal{W}(s) - \mathcal{W}_N(s) \| _1 + 1  ).
\end{equation}
For $s \in K$,  \eqref{eq:trace_norm_of_difference} yields that the right hand side \eqref{estimate_on_F} is bounded by a constant depending on $K$.  Hence, the function $F(s)$ is a uniform limit of holomorphic functions, that implies by Morera's theorem its own holomorphy. By \cite[(36)]{soares2022hecke}, for $\Re(s)>1 / 2$ the function $F(s)$ coincides with the Selberg zeta function  $Z_{\Gamma_w} (s)$, that together with uniqueness of analytic continuation  finishes the proof. 
\end{proof}

\begin{remark}
The estimate \eqref{eq:approximation_of_selberg_zeta} is not good for large values of $\operatorname{Im}(s)$ -- specifically, it grows like $\exp (\exp (\operatorname{Im}(s) )))$. Consequently, to maintain accuracy, $N$ must be chosen of order $\exp (\operatorname{Im}(s))$.
    
    The bottleneck of the proof is in the fact that in  \eqref{estimate_on_F}, we need to find the exponent of $\| \mathcal{W}(s) \|_1 $, and the latter already exhibits exponential growth in $|\operatorname{Im}(s)|$.  However, other examples in the literature of similar estimates (that the author is aware of) use an estimate of singular values  that also depends on $\operatorname{Im}(s)$ exponentially; cf. \cite[Proposition 2.3]{naud2018fractal} or    \cite[discussion above Lemma 3.3]{guillope2004selberg}.
\end{remark}

\section{Numerical estimates of Hausdorff dimensions}\label{sec:haus_calc}

To find roots of $Z(s)$ that belong to the real line, we can use the bisection method. 

More precisely, note that $\zeta(s)$ is a real-valued function for $s \in \mathbb{R} \setminus \{ 1 \}$, that implies the function $F(s)$ is real-valued as well. To prove that there is a root that belongs to an interval $[s_1, s_2] \subset \mathbb{R}$ (not containing poles of $F(s)$), it suffices to show $F(s_1)\cdot F(s_2) <0$. To prove, in turn, that $F(s_1) < 0$, we write $
    F(s_1) = F_N(s_1) + (F(s_1)-F_N(s_1))$ and note it will be sufficient to find $N \in \mathbb{N}$ such that 
\begin{enumerate}[i)]
    \item $F_N(s_1) < 0$ and 
    \item $|F(s_1)  - F_N(s_1) | < |F_N(s_1)|$, that can be evaluated as the right hand side of \eqref{eq:approximation_of_selberg_zeta} where instead of $s$ we substitute $s_1$.
\end{enumerate}

For example, take $s_1 = 0.75$, $w = 3$. For $|F(s_1)-F_N(s_1)|$ to be of order $10^{-50}$,  it suffices to take $N=350$. For larger $w$, smaller $N$ will suffice: for example, let $s_1=0.55$ and  $w=8$, then for $|F(s_1)-F_N(s_1)|$ to be of order $10^{-50}$, it will suffice to take $N=100$. 

Now that we have the method of proving that in a given interval, there exists a zero of $Z(s)$, we can start with an interval where such a zero exists a priori; e.g., from \cite[Theorem 1.4]{soares2022hecke} and \cite[Theorem 6.1]{phillips1985laplacian}, the Selberg zeta function $Z_{\Gamma_w}(s)$ always has a zero between $1/2$ and $1$. Then, one improves the precision with the help of a bisection method. 

The result of the calculations, performed with the help of \cite{PARI2}, are present in Table \ref{tab:hausdorff}. Additionally, for $w=2 \pi$, the Hausdorff dimension can be approximated by the the reciprocal of the golden ratio with an approximate error of $3 \cdot 10^{-4}$ (though, of course, this is merely a coincidence).

\section{Proof of Corollary \ref{sec:ruelle_hecke_triangle} (Ruelle zeta function at zero)}

The idea of the proof is to approximate $R_{\Gamma_w}(0)$ by a ratio of determinants of finite matrices $1- \mathcal{V}_N(0)$ and $1- \mathcal{V}_{N-1}(1)^T$ (to be defined later). We will see that \(1 - \mathcal{V}_N(0)\) admits an upper block-diagonal structure, where one of the blocks coincides with the diagonal transpose \(1 - \mathcal{V}_{N-1}(1)^T\) of the other matrix.

For $s \in \mathbb{C} \setminus \frac{1}{2} (1- 2 \mathbb{N}_0)$ and  $i,j \in \mathbb{N}_0$, let 
\begin{equation}\label{eq:aij}
    a_{ij}(s) = \sqrt{\frac{i+1}{j+1}} \ell_{ij}(s) \stackrel{\eqref{eq:elsijeven}, \eqref{eq:elsijodd}}{=} 
    \begin{cases}
       2  \frac{\zeta(2 s+i+j)}{w^{2 s+i+j}}      \binom{2 s+i+j-1}{i}, &\quad i+j \text{ even},\\
       0, &\quad i+j \text{ odd}.
    \end{cases}
\end{equation}
For $N \in \mathbb{N}$, let  $\mathcal{V}_N(s)$ be defined similarly to \eqref{def:mathcalWs} with the only exception that we replace all occurances of $\ell_{ij}(s)$ by $a_{ij}(s)$. That is, the operator $\mathcal{V}_N(s)$ can be represented by an $(N \times N)$-matrix with the  $(i,j)$-th entry equal to $a_{ij}(s)$ instead of $\ell_{ij}(s)$. The respective determinants do not change, 
\begin{equation}\label{eq:VandWareSame}
\det (1 - \mathcal{V}_N(s)) = \det (1 - \mathcal{W}_N(s)).     
\end{equation}
Note for $i \in \mathbb{N}$, we have  $a_{i0}(0) = 0$ from the properties of binomial coefficients, thus every $(i,0)$-th entry of $\mathcal{V}_N(0)$ vanishes. Hence, the matrix $\mathcal{V}_{N}(0)$ is an upper-block diagonal matrix, where the first block is equal to an $(1 \times 1)$-matrix with the single entry 
\begin{equation}\label{eq:a00entry}
    a_{00}(0) =  2\zeta(0) = -1.
\end{equation} 
Additionally, from \eqref{eq:aij} and 
\begin{equation}
    \binom{(i+1)+(j+1)-1}{i+1} =    \binom{2+i+j-1}{j},
\end{equation}
one has 
\begin{equation}
    a_{i+1,j+1}(0) =     a_{ji}(1),
\end{equation}
thus the $(i+1,j+1)$-st entry of  $\mathcal{V}_{N}(0)$ coincides with the $(j,i)$-th entry of $\mathcal{V}_{N-1}(1)$, that together with \eqref{eq:a00entry} implies  
\begin{equation}
\mathcal{V}_{N}(0) \;=\; 
\begin{pmatrix}
-1 & \ldots \\
0 & \mathcal{V}_{N-1}(1)^{T}
\end{pmatrix}    .
\end{equation}
From this, we obtain an equality for the determinants,
    \begin{equation}
\det (1 - \mathcal{V}_N(0))=2 \det (1 - \mathcal{V}_{N-1}(1)),
    \end{equation}
    and, from \eqref{eq:VandWareSame} and  \eqref{eq:FN_definition}, we get 
    \begin{equation}
        F_{N}(0)=2 F_{N-1}(1).
    \end{equation}
By Theorem \ref{thm:main_theorem_1}, the Selberg zeta function $Z_{\Gamma_w}(s)$ is approximated by $F_N(s)$, that together with \eqref{eq:def_Ruelle_zeta} finishes the proof.

\begin{remark}
    It may at first appear that our result is not consistent with the works of \cite{benard2023twisted}, \cite{dyatlov2017ruelle}, or \cite{frahm2023twisted}. In those references, the behavior of the Ruelle zeta function at the origin is described in terms of the Euler characteristic $\chi(X)$ of the surface $X$, namely
    \begin{equation}\label{eq:Ruelle_aaaaa}
        R_X(s) \sim s^{-\chi(X)}, \quad s \to 0,
    \end{equation}
    that is not the case of the article. However, in our case, the Euler characteristic of the surface $\Gamma_w \backslash \mathbb{H}$ is not an integer, and a literal application of \eqref{eq:Ruelle_aaaaa} would suggest a non-integer exponent, hence a non-analytic behavior of the Ruelle zeta function at the origin.
\end{remark}

\section{Trivial zeroes of $Z_{\Gamma_w}(s)$}
An argument similar to the one in the previous section allows to evaluate the orders of trivial zeros of $Z_{\Gamma_w}(s)$. 
\begin{proof}[Proof of Corollary \ref{lem:skdjfjfjfjsllfksdjlkflskdflksdflksjdfjsdjfjfjjj}]
    For $N \in \mathbb{N}$, we let $\mathcal{V}_N(s)$ be as in the proof of Corollary~\ref{sec:ruelle_hecke_triangle}.  From the definition of binomial coefficients, we get 
    \begin{equation}
    a_{ij}(-m) = 0 \quad \text{ for }  0 \le  j  \le 2m \text{ and } \, i \ge 2m+1,
    \end{equation}
hence $\mathcal{V}_{N+2m+1}(-m)$ is a block matrix.   
Moreover, 
\begin{equation}\label{eq:skkslskdkkf}
    a_{i+m+1,j+m+1}(-m) =     a_{ji}(1+m), \quad i, j \ge 0.
\end{equation}
Note that the right hand side of \eqref{eq:skkslskdkkf} is an $(i,j)$-th entry of the transpose of $\mathcal{V}_N(1+m)$, and the left hand side of \eqref{eq:skkslskdkkf} is the $(i+m+1,j+m+1)$-th entry of $\mathcal{V}_{N+2m+1}(-m)$. Thus, the matrix $\mathcal{V}_{N+2m+1}(-m)$  splits in blocks as follows:
\begin{equation}\label{eq:aasppppppspapspsp3332}
\mathcal{V}_{N+2m+1}(-m) \;=\; 
\begin{pmatrix}
U(0) & \ldots \\
0 & \mathcal{V}_N(1+m)^T
\end{pmatrix}   
\end{equation}
where 
\begin{equation}
 U(s) = ( a_{ij}(-m+s) )_{i,j=0}^{2m}.   
\end{equation}
The analyticity of $a_{ij}(s)$ together with \eqref{eq:aasppppppspapspsp3332} implies that as $s \to 0$, 
\begin{equation}
\mathcal{V}_{N+2m+1}(-m+s) \;=\; 
\begin{pmatrix}
U(s)  & \ldots \\
V(s) & \mathcal{V}_N(1+m)^T + O(s)
\end{pmatrix} ,  
\end{equation}
where $V(s)$ is an $N \times (2m+1)$-dimensional matrix with entries of order $O(s)$ as $s \to 0$, 
thus 
\begin{equation}\label{wq:sdodsidifdfdjsoidfsodfsdfs}
\begin{split}
&    \det (1-\mathcal{V}_{N+2m+1}(-m+s)) \\
& \ \ \ \ \   \ \ \ \ \ \ \ \ \ \ \ \ \ \ =  \det (1-\mathcal{V}_{N}(1+m)) \cdot \det (1-U(s)) + O(s^N).
\end{split}
\end{equation}
From the properties of $a_{ij}(-m+s)$, the matrix $1 - U(0)$ has nonzero 
entries only on its diagonal and anti-diagonal. If $m$ is even, all positions on these two diagonals are nonzero; if $m$ is odd, the central entry at their intersection vanishes.
Moreover, for $j=0, 1, \ldots, m$, the $j$-th  and $(2m+1-j)$-th columns in $1-U(0)$ are linearly dependent. That implies $1-U(0)$ has rank $m+1$ for even $m$ and rank $m$ for odd $m$.  For example, for $m=1$, we get 
\begin{equation}
1-U(0) =     
\begin{pmatrix}
 1 & 0 & 1 \\
 0 & 0 & 0 \\
 1 & 0 & 1 \\
\end{pmatrix}
\end{equation}
is of rank $1$, and 
in  the case $m=2$,  we have 
\begin{equation}
1-U(0) = 
\begin{pmatrix}
 1 & 0 & 0 & 0 & 1 \\
 0 & 1 & 0 & -1 & 0 \\
 0 & 0 & 2 & 0 & 0 \\
 0 & -1 & 0 & 1 & 0 \\
 1 & 0 & 0 & 0 & 1 \\
\end{pmatrix}
\end{equation}
is of rank $3$.

Thus, $\det (1-U(s))$ has a zero at $s=0$ of degree between 
\begin{equation}
\dim ( 1- U(0)) - \operatorname{rank}(1-U(0)) \ge (2m+1)-(m+1)=m    
\end{equation}
and $2m+1$. On the other hand, as $N \to \infty$, the determinant $\det (1-\mathcal{V}_{N}(1+m))$  approximates $Z_{\Gamma_w}(m+1)$, and the latter is not equal to zero as a convergent product~\eqref{eq:sssppppp2p2p2pp2p2}. Thus for some sufficiently large  $N$,  $\det (1-\mathcal{V}_{N}(1+m)) \neq 0$.

Hence, for large values of $N$, $\det (1-\mathcal{V}_{N}(1+m)) \cdot \det (1- U(s))$ has a zero at $s=0$ of a degree between $m$ and $2m+1$. Together with \eqref{wq:sdodsidifdfdjsoidfsodfsdfs}, it finishes the proof.
\end{proof}

\begin{remark}\label{rem:strand}
The proof of the lemma above illustrates the limitation of approximating $Z_{\Gamma_w}(s)$ by $F_N(s)$ for $\Re(s) \ll 0$. Namely, that it shows that for $s=-m$ to be a zero of $F_N(s)$, we need to take $N$ to be at least equal to $m+1$, else $F_N(-m)$ will have the full rank, and $\det F_N(-m) \neq 0$.
\end{remark}

\bibliography{bib} 

\begin{bibdiv}
\begin{biblist}

\bib{benard2023twisted}{article}{
      author={B{\'e}nard, L.},
      author={Frahm, J.},
      author={Spilioti, P.},
       title={The twisted {R}uelle zeta function on compact hyperbolic
  orbisurfaces and {R}eidemeister--{T}uraev torsion},
        date={2023},
     journal={Journal de l’{\'E}cole polytechnique—Math{\'e}matiques},
      volume={10},
       pages={1391\ndash 1439},
}

\bib{dang2020fried}{article}{
      author={Dang, N.~V.},
      author={Guillarmou, C.},
      author={Riviere, G.},
      author={Shen, S.},
       title={The {F}ried conjecture in small dimensions},
        date={2020},
     journal={Inventiones mathematicae},
      volume={220},
      number={2},
       pages={525\ndash 579},
}

\bib{NIST}{misc}{
       title={{\it NIST Digital Library of Mathematical Functions}},
         how={http://dlmf.nist.gov/, Release 1.1.6 of 2022-06-30},
         url={http://dlmf.nist.gov/},
        note={F.~W.~J. Olver, A.~B. {Olde Daalhuis}, D.~W. Lozier, B.~I.
  Schneider, R.~F. Boisvert, C.~W. Clark, B.~R. Miller, B.~V. Saunders, H.~S.
  Cohl, and M.~A. McClain, eds.},
}

\bib{dyatlov2017ruelle}{article}{
      author={Dyatlov, S.},
      author={Zworski, M.},
       title={Ruelle zeta function at zero for surfaces},
        date={2017},
     journal={Inventiones mathematicae},
      volume={210},
      number={1},
       pages={211\ndash 229},
}

\bib{fedosova2020meromorphic}{article}{
      author={Fedosova, K.},
      author={Pohl, A.},
       title={Meromorphic continuation of {Selberg} zeta functions with twists
  having non-expanding cusp monodromy},
        date={2020},
     journal={Selecta Mathematica},
      volume={26},
      number={1},
       pages={9},
}

\bib{frahm2023twisted}{article}{
      author={Frahm, J.},
      author={Spilioti, P.},
       title={Twisted {R}uelle zeta function at zero for compact hyperbolic
  surfaces},
        date={2023},
     journal={Journal of Number Theory},
      volume={243},
       pages={38\ndash 61},
}

\bib{fried1986analytic}{article}{
      author={Fried, D.},
       title={Analytic torsion and closed geodesics on hyperbolic manifolds},
        date={1986},
     journal={Inventiones mathematicae},
      volume={84},
      number={3},
       pages={523\ndash 540},
}

\bib{gohberg2012traces}{book}{
      author={Gohberg, I.},
      author={Goldberg, S.},
      author={Krupnik, N.},
       title={Traces and determinants of linear operators},
   publisher={Birkh{\"a}user},
        date={2012},
      volume={116},
}

\bib{guillope2004selberg}{article}{
      author={Guillop{\'e}, L.},
      author={Lin, K.~K.},
      author={Zworski, M.},
       title={The {S}elberg zeta function for convex co-compact {S}chottky
  groups},
        date={2004},
     journal={Communications in mathematical physics},
      volume={245},
       pages={149\ndash 176},
}

\bib{jorgenson2024functional}{article}{
      author={Jorgenson, J.},
      author={Lee, M.},
      author={Smajlovic, L.},
       title={On the functional equation of twisted {R}uelle zeta function and
  {F}ried's conjecture},
        date={2024},
     journal={arXiv preprint arXiv:2402.02959},
}

\bib{mayer1991thermodynamic}{article}{
      author={Mayer, D.~H.},
       title={The thermodynamic formalism approach to {Selberg's} zeta function
  for $\operatorname{PSL} (2,\mathbb{Z})$},
        date={1991},
}

\bib{moller2013period}{article}{
      author={M{\"o}ller, M.},
      author={Pohl, A.~D.},
       title={Period functions for {H}ecke triangle groups, and the {Selberg}
  zeta function as a {F}redholm determinant},
        date={2013},
     journal={Ergodic Theory and Dynamical Systems},
      volume={33},
      number={1},
       pages={247\ndash 283},
}

\bib{moscovici1991r}{article}{
      author={Moscovici, H.},
      author={Stanton, R.~J.},
       title={{R}-torsion and zeta functions for locally symmetric manifolds},
        date={1991},
     journal={Inventiones mathematicae},
      volume={105},
      number={1},
       pages={185\ndash 216},
}

\bib{mueller2020fried}{article}{
      author={Mueller, W.},
       title={On {F}ried's conjecture for compact hyperbolic manifolds},
        date={2020},
     journal={arXiv preprint arXiv:2005.01450},
}

\bib{naud2018fractal}{article}{
      author={Naud, F.},
      author={Pohl, A.},
      author={Soares, L.},
       title={Fractal weyl bounds and {H}ecke triangle groups},
        date={2018},
     journal={arXiv preprint arXiv:1810.04489},
}

\bib{phillips1985laplacian}{article}{
      author={Phillips, R.},
      author={Sarnak, P.},
       title={The {L}aplacian for domains in hyperbolic space and limit sets of
  {K}leinian groups},
        date={1985},
}

\bib{pohl2015thermodynamic}{article}{
      author={Pohl, A.~D.},
       title={A thermodynamic formalism approach to the {Selberg} zeta function
  for {H}ecke triangle surfaces of infinite area},
        date={2015},
     journal={Communications in Mathematical Physics},
      volume={337},
      number={1},
       pages={103\ndash 126},
}

\bib{pohl2016symbolic}{article}{
      author={Pohl, A.~D.},
       title={Symbolic dynamics, automorphic functions, and {Selberg} zeta
  functions with unitary representations},
        date={2016},
     journal={Contemp. Math},
      volume={669},
       pages={205\ndash 236},
}

\bib{pollicott1991some}{article}{
      author={Pollicott, M.},
       title={Some applications of thermodynamic formalism to manifolds with
  constant negative curvature},
        date={1991},
     journal={Advances in mathematics},
      volume={85},
      number={2},
       pages={161\ndash 192},
}

\bib{pollicott1997remarkable}{article}{
      author={Pollicott, M.},
      author={Rocha, A.~C.},
       title={A remarkable formula for the determinant of the {L}aplacian},
        date={1997},
     journal={Inventiones mathematicae},
      volume={130},
       pages={399\ndash 414},
}

\bib{sarnak2008spectrum}{article}{
      author={Sarnak, P.},
      author={Phillips, R.},
       title={On the spectrum of the {H}ecke groups},
        date={2008},
     journal={Duke Math. J},
      volume={52},
       pages={211\ndash 221},
}

\bib{shen2017analytic}{article}{
      author={Shen, S.},
       title={Analytic torsion, dynamical zeta functions, and the {F}ried
  conjecture},
        date={2017},
     journal={Analysis \& Pde},
      volume={11},
      number={1},
       pages={1\ndash 74},
}

\bib{soares2022hecke}{inproceedings}{
      author={Soares, L.},
       title={{H}ecke triangle groups, transfer operators and {H}ausdorff
  dimension},
organization={Springer},
        date={2022},
   booktitle={{A}nnales {H}enri {P}oincar{\'e}},
       pages={1\ndash 43},
}

\bib{PARI2}{manual}{
       title={{PARI/GP version \texttt{13.2.0}}},
organization={{The PARI~Group}},
     address={Univ. Bordeaux},
        date={2024},
        note={available from \url{http://pari.math.u-bordeaux.fr/}},
}

\end{biblist}
\end{bibdiv}
\bibliographystyle{amsplain}

\end{document}